\documentclass[12pt]{article}
\usepackage[T2A]{fontenc}
\usepackage[cp866]{inputenc}
\usepackage[tbtags]{amsmath}
\usepackage{amsfonts,amssymb,mathrsfs,amscd, amsthm,comment}
\usepackage{graphicx}
\usepackage{datetime}
\usepackage{mathrsfs}

\usepackage[breaklinks]{hyperref}

\hypersetup{
unicode=true,
colorlinks=true,
linkcolor=blue,
citecolor=blue,
urlcolor=blue,
filecolor=blue,
bookmarksnumbered=true,
pdfstartview=FitH,
pdfhighlight=/N
}

\newdimen\normalparindent
\newtheorem*{theorem*}{Theorem}
\newtheorem{theorem}{Theorem}
\newtheorem{Definition}{Definition}
\newtheorem{Corollary}{Corollary}
\newtheorem{Proposition}{Proposition}
\newtheorem{Statement}{Statement}
\newtheorem{Lemma}{Lemma}
\newtheorem{Question}{Question}

\newtheorem{Example}{Example}
\newtheorem{Remark}{Remark}

\def\AA{\mathcal A_\infty}

\def\TT{\mathbb T}

\let\myh\widehat
\let\myt\widetilde
\let\myo\overline

\def\CC{\mathbb C}
\def\NN{\mathbb N}
\def\ZZ{\mathbb Z}

\def\RRR{\mathbf R}
\def\GGG{\mathbf G}
\def\MM{\mathscr M}

\def\DD{\mathbb D}
\def\KKf{\mathcal K_{f_\infty}}
\def\ddc{\operatorname{dd^c}}

\def\pr{\operatorname{pr}}
\def\mcap{\operatorname{cap}}
\def\supp{\operatorname{supp}}
\def\Meas{\operatorname{Meas}}

\def\dist{\operatorname{dist}}
\def\diam{\operatorname{diam}}

\def\const{\operatorname{const}}
\def\Meas{\operatorname{Meas}}

\begin{document}


\title{\bf{
Existence of a ``maximal'' domain of meromorphy for an
analytic function outside a polar compact set}}

\author{A.~V.~Komlov}

\date{}

\maketitle

\markright{}


\begin{abstract}
Let E be a polar compact set in $\CC$. Let $f_\infty$ be a germ at $\infty$ that can be analytically continued along an arbitrary path $\gamma$ lying in $\myh\CC\setminus E$ and starting at the point $\infty$. 
In 1985--1986 Herbert Stahl presented his proof of a fundamental theorem on the convergence of diagonal Pad\'e approximants constructed from such germ $f_\infty$.
Since then, this theorem has borne his name.
A crucial role in his proof is played by the existence of a compact set $S_{f_\infty}$ of minimal logarithmic capacity among all compact sets $K$ such that the germ $f_\infty$ extends as a single-valued meromorphic function to $\myh\CC\setminus K$.
Unfortunately, the proof of this fact presented by H. Stahl in 1985 contains a crucial mistake. It is surprising that this mistake was made in the original Stahl's paper in 1985 and repeated in his last preprint in 2012, and, as far as we know, no one has pointed it out before! In this paper we explain this serious Stahl's mistake and present a correct proof of the existence of a compact set $S_{f_\infty}$. We emphasize that our proof will use ideas completely different from Stahl's ideas. Also we discuss the more general Stahl's conjecture about the existence of a compact set $S_{f_\infty}$ for an arbitrary germ $f_\infty$ without any assumption on the paths along that the germ $f_\infty$ can be continued. 
\end{abstract}

\setcounter{tocdepth}{1}\tableofcontents

\makeatletter
\renewcommand{\@makefnmark}{}
\makeatother

\section{Introduction}
\label{s1}

Let E be a polar compact set in $\CC$.
We denote the Riemann sphere by $\myh\CC$.
Let us denote the set of all analytic germs $f_\infty$ at $\infty$ that can be meromorphically continued along an arbitrary path $\gamma$ lying in $\myh\CC\setminus E$ and starting at the point $\infty$ and that don't continue to $\myh\CC\setminus E$ as a (single-valued) meromorphic function by $\AA(\myh\CC\setminus E)$. Always when we say about analytic continuation along a path, we mean meromorphic continuation along it.
In the series of five works \cite{St85-1}, \cite{St85-2}, \cite{St85-3}, \cite{St86-1}, \cite{St86-2} H. Stahl presented  the proof of Stahl's theorem on the weak asymptotic of (diagonal) Pad\'e polynomials and the convergence of (diagonal) Pad\'e approximants constructed from such germ $f_\infty\in \AA(\myh\CC\setminus E)$. 
Note that if $f$ is a meromorphic function in $\myh\CC\setminus E$, then Pad\'e approximants for $f$ converge in capacity to f on each compact set $K\subset\CC$, see~\cite{Pom73}.
To formulate Stahl's theorem we need to recall some notions and introduce some notation.

For each measure  $\mu$ in $\CC$ we denote its support by $\supp \mu$. 
For each set $G$ in $\CC$ we denote the set of all non-negative Borel measures with support on $G$ by $\Meas(G)$ and the set of all non-negative Borel unit measures with support on $G$ by $\Meas_1(G)$, i.e. $\mu\in\Meas(G)$ if $\supp \mu\subset G$ and $\nu\in\Meas_1(G)$ if $\supp \nu\subset G$ and $\nu(\CC)=1$.
We denote weak* convergence in the space of measures on $\CC$ by the symbol $\xrightarrow{*}$.

For each measure $\mu$ with compact support we define its logarithmic potential $V^{\mu}$ and its energy $E(\mu)$ as
\begin{equation}
\label{VE}
V^{\mu}(z):= \int\limits_\CC \log|z-t|\mu(t), \qquad
E(\mu):= \int\limits_\CC V^{\mu}(z)\mu(z).
\end{equation}
Often the sign ``minus'' is used in the definition of $V^{\mu}$, but we define it as in~\eqref{VE}. Thus, $V^{\mu}$ is a subharmonic function in $\CC$ and the energy of a delta-measure is $-\infty$. In such definition we follow Ransford's book~\cite{Ra95}, where all facts of the logarithmic potential theory that we are now presenting are contained.

For each compact set $K$ we define its energy $E(K)$ and its (logarithmic) capacity $\mcap K$ as
\begin{equation}
\label{E(K)}
E(K):=\sup\limits_{\mu\in\Meas_1(K)}E(\mu), \qquad
\mcap K:=e^{E(K)}.
\end{equation}
If $E(K)=-\infty$ ($\mcap K=0$) then $K$ is called polar.
It is well known that if $K$ is a non-polar compact set then there exists a unique measure $\lambda_K\in\Meas_1(K)$ such that $E(K)=E(\lambda_K)$. This measure $\lambda_K$ is called the equilibrium measure of the compact set $K$. Due to classical Frostman's theorem $\lambda_K$ is equivalently defined as the unique measure $\lambda_K\in\Meas_1(K)$ that satisfies the following equilibrium condition:
\begin{equation}
\label{V^lambda}
\begin{aligned}
V^{\lambda_K}(z) &=\const=:-\gamma_K \quad \mbox{n. e. on } K,\\
V^{\lambda_K}(z) &> -\gamma_K \quad \mbox{on } \CC\setminus K,
\end{aligned}
\end{equation}
where n.e. denotes nearly everywhere, i.e. up to a polar set. Moreover, it immediately follows from the definition~\eqref{VE} that $E(K)=-\gamma_K$.

For each non-polar compact set $K$ we denote the Green function of its complement $\CC\setminus K$ with the logarithmic singularity at $\infty$ by $g_{\myh\CC\setminus K}(z)$. Recall that $g_{\myh\CC\setminus K}(z)$ is the subharmonic function in $\CC$, harmonic in $\CC\setminus K$, $g_{\myh\CC\setminus K}(z)=0$ n. e. on $K$ and
\begin{equation}
\label{g(z)}
g_{\myh\CC\setminus K}(z)=\log |z|+\gamma_K+o(1) \quad \mbox{as } z\to\infty.
\end{equation}
The constant $\gamma_K$ is called the Robin constant. From the definition it immediately follows that $\gamma_K$ in~\eqref{g(z)} and~\eqref{V^lambda} is the same and
\begin{equation}
\label{g-V}
g_{\myh\CC\setminus K}(z)=V^{\lambda_K}(z)+\gamma_K.
\end{equation}

Recall the definition of convergence in capacity. We say that a sequence of functions $\varphi_n$ converges to a function $\varphi$ in capacity on a set $G$ if for any $\varepsilon>0$ we have
\begin{equation}
\label{cap)}
\lim\limits_{n\to\infty}\mcap\{z\in G: |\varphi_n(z)-\varphi(z)|>\varepsilon\}=0.
\end{equation}
We denote convergence in capacity by $\varphi_n\xrightarrow{\mcap}\varphi$.

Now recall the definition of (diagonal) Pad\'e polynomials. The (diagonal) Pad\'e polynomials of order $n$ constructed from a given germ $f_\infty$ at $\infty$ are two polynomials $P_{n,0}, P_{n,1}$ such that their degrees are not greater than $n$, $P_{n, j}\not\equiv 0$, and
\begin{equation*}
\label{Pade}
P_{n,0}(z)-P_{n,1}(z)f_\infty(z)=O\left(\frac{1}{z^{n+1}}\right)\quad \mbox{as } z\to\infty.
\end{equation*}
The fraction $P_{n,0}/P_{n,1}$ is called the Pad\'e approximant of order $n$.

For each polynomial $P$ we denote the counting measure of its zeroes by $\chi(P)$, i.e. 
\begin{equation*}
\label{chi(P)}
\chi(P)=\sum\limits_{x:P(x)=0}\delta_x,
\end{equation*}
where $\delta_x$ is the delta-measure at $x$.

For each domain $D$ we denote the space of (single-valued) meromorphic functions in $D$ by $\MM(D)$. If a given germ $f_\infty$ at $\infty$ is continued in some domain $D\ni\infty$ as a (single-valued) meromorphic function we write $f_\infty\in\MM(D)$. Now we introduce a key notion for us.

\begin{Definition}
\label{Def_KKf}
For a given analytic germ $f_\infty$ at $\infty$ we call a compact set $K\subset\CC$ admissible for $f_\infty$ if $\myh\CC\setminus K$ is connected and $f_\infty\in\MM(\myh\CC\setminus K)$. We denote the set of all admissible compact sets for $f_\infty$ by $\KKf$. 
\end{Definition}

Recall that for each compact set $E$ the space $\AA(\myh\CC\setminus E)$ is the set of all analytic germs $f_\infty$ at $\infty$ that can be meromorphically continued along an arbitrary path $\gamma$ lying in $\myh\CC\setminus E$ and starting at the point $\infty$ and such that $f_\infty\notin\MM(\myh\CC\setminus E)$.
Now we can formulate Stahl's theorem.

\begin{theorem*}[Stahl, 1985--1997]
Let $E$ be a polar compact set. Let a germ $f_\infty\in \AA(\myh\CC\setminus E)$. Let $P_{n,0}, P_{n,1}$ be Pad\'e polynomials constructed from $f_\infty$. Then

1) There exists a compact set $S\in\KKf$ such that $\mcap S = \inf\limits_{K\in\KKf}\mcap K$. Moreover, for any two such minimal compact sets $S_1, S_2$ we have $\mcap(S_1\setminus S_2)=\mcap(S_2\setminus S_1)=0$ and the intersection of all such minimal compact sets is also a minimal compact set $S^\circ\in\KKf$. 
In particular, $f_\infty\in\MM(\myh\CC\setminus S^\circ)$.

2) $\dfrac{1}{n}\chi(P_{n,j})\xrightarrow{*}\lambda_{S^\circ}$, where $\lambda_{S^\circ}$ is the equilibrium measure of $S^\circ$.

3) $\dfrac{P_{n,0}}{P_{n,1}}\xrightarrow{\mcap} f_\infty$ on each compact set $F\subset \CC\setminus S^\circ$.

4) $\dfrac{1}{n}\log\left|\dfrac{P_{n,0}}{P_{n,1}}-f_\infty\right|\xrightarrow{\mcap}e^{-2g_{\myh\CC\setminus S^\circ}}$
on each compact set $F\subset \CC\setminus S^\circ$.
\end{theorem*}

The scheme of Stahl's proof of this theorem is the following.
In~\cite{St85-1} Stahl presented the proof of existence of a minimal compact set $S\in\KKf$. Moreover, his proof was given in the most general case, i.e. for an arbitrary germ $f_\infty$ without any assumption on the paths along that the germ $f_\infty$ can be continued. (We discuss it in details in \S\ref{s2}.) In~\cite{St85-2} Stahl presented the proof of the rest part of the item 1) of Stahl's theorem, in particular, the existence of the minimal compact set $S^\circ$. In~\cite{St85-3} Stahl presented the proof of the S-property of $S^\circ$. The S-property for $S^\circ$ means that there exists the polar compact set $E^\circ\subset S^\circ$ such that $S^\circ\setminus E^\circ$ is a union of open non-intersecting analytic arcs $\gamma_\alpha$, and on each arc $\gamma_\alpha$ we have $\frac{\partial g_{\myh\CC\setminus S^\circ}}{\partial n_+}=\frac{\partial g_{\myh\CC\setminus S^\circ}}{\partial n_-}$, where $\frac{\partial g_{\myh\CC\setminus S^\circ}}{\partial n_\pm}$ are normal derivatives of Green function $g_{\myh\CC\setminus S^\circ}$ from different sides of $\gamma_\alpha$.
The S-property plays the crucial role for Stahl's proof of the items 2)--4) of his theorem.
In~\cite{St86-1} Stahl presented the proof of item 2) of his theorem.  In~\cite{St86-2} Stahl described the asymptotic behaviour of some integral expressions containing $P_n$, $Q_n$ and $f_\infty$. But the items 3), 4) of Stahl's theorem are not contained in~\cite{St86-2} and do not directly follow from results of~\cite{St86-2}. It is clear that 3) follows from 4) in Stahl's theorem. Actually Stahl presented the proof of the item 4) only in 1997 in~\cite{St97}.

Unfortunately Stahl's proof of existence of a minimal compact set $S\in\KKf$ has a crucial mistake and it is wrong. Thus, it is the mistake on the first step of Stahl's proof of his theorem. This mistake was made in original paper~\cite{St85-1}, in~\cite{St97} Stahl referred to~\cite{St85-1}, and he repeated this mistake in~\cite{St12}. We explain this mistake in details in \S\ref{s2}. It should be noted that in 1994 Perevoznikova and Rakhmanov~\cite{PeRa} proved the existence of a minimal compact set $S\in\KKf$ in the more simple case when the set $E$ consists of a finite number of points in $\CC$ and $f_\infty\in\AA(\myh\CC\setminus E)$. Unfortunately, this preprint was never published. But proofs and results of it were represented in many papers. The most full presentation was given in~\cite[Section~3]{AptYa}. The main result of our paper is the following theorem, which gives a correct proof of the first step of Stahl's proof of his theorem.

\begin{theorem}
\label{t1}
Let $E$ be a polar compact set in $\CC$. Let $f_\infty\in\AA(\myh\CC\setminus E)$. Then there exists a compact set $S\in\KKf$ such that $\mcap S = \inf\limits_{K\in\KKf}\mcap K$.
\end{theorem}

Note that our proof use ideas completely different from Stahl's ideas and in some sense our proof is a generalization of the proof of  Perevoznikova and Rakhmanov.
We use some statements from~\cite{PeRa} and~\cite{St12} and for the sake of completeness, we present proofs of these statements in \S\ref{s3}. 
Also in \S\ref{s2} we discuss the most general analogue of Theorem~\ref{t1} when $f_\infty$ is an arbitrary germ at $\infty$ (without any assumption on the paths along that the germ $f_\infty$ can be continued). We emphasize that Stahl tried to prove namely this analogue of Theorem~\ref{t1}. Nevertheless, Theorem~\ref{t1} is sufficient for the first step of Stahl's proof of his theorem. Note that Stahl's proof has also gaps in other steps.

The article is organized as follows. In \S\ref{s2} we explain the mistake in Stahl's proof of the analogue of Theorem~\ref{t1} and formulate several interesting open questions that follows from his reasoning. In \S\ref{s3} we prove many preliminary results. One part of these results is formulated in~\cite{St12}, but their proofs very often contain gaps and sometimes are incorrect. We give their full proofs. Another part of these results is contained in~\cite{PeRa}, we also present their proofs for the completeness of the presentation.
In the end of \S\ref{s3} we present additional preliminary results, which we will use for the proof of Theorem~\ref{t1}.
In \S\ref{s4} we present a proof of Theorem~\ref{t1}.

\section{Stahl's mistake and open questions}
\label{s2}

In~\cite[Theorem 1]{St85-1} and in~\cite[Proposition 6]{St12} Stahl tried to prove the following statement, but made a mistake. We formulate this statement as an open question.
\begin{Question}
\label{Q1}
Is the following true?
Let $f_\infty$ be an arbitrary analytic germ at $\infty$. Then there exists a compact set $S\in\KKf$ such that $\mcap S = \inf\limits_{K\in\KKf}\mcap K$.
\end{Question}
As we see, this is a more general statement than in our Theorem~\ref{t1}. Now we explain Stahl's mistake. He did the following.

Let $\{K_n\}$ be a minimizing sequence in the Question~\ref{Q1}, i.e. $K_n\in\KKf$ and
\begin{equation}
\label{c_0}
\lim_{n\to\infty}\mcap K_n = \inf\limits_{K\in\KKf}\mcap K =: c_0.
\end{equation}

Stahl says that we can suppose that all $K_n$ belong to some big (closed) disc $\myo\DD(0,R)$ with center at 0 and of radius $R$. Stahl formulates this in~\cite[Lemma~4]{St85-1} and in~\cite[Lemma~5]{St12}. But the proof of~\cite[Lemma~4]{St85-1} has a gap and the proof of~\cite[Lemma~5]{St12} is incomplete and repeats a part of the proof of~\cite[Lemma~4]{St85-1}. We give a full proof of this statement in Proposition~\ref{Prop_D}, see also Corollary~\ref{Cor_D}. So, let all $K_n\subset \myo\DD(0,R)$.

Then Stahl proves (with gaps), see ~\cite[Lemma~4]{St12}, that if $c_0=0$, then $f_\infty\in\MM(\myh\CC\setminus E_0)$, where $E_0$ is a polar compact set. 
We give a full proof of this fact in Statement~\ref{St_c_0}. Thus, if $c_0=0$, we have $S:=E_0$ and $\mcap S =0$. Therefore, further we suppose that $c_0>0$.

Now let $\lambda_{K_n}$ be the equilibrium measure of $K_n$. Since all $K_n\subset \myo\DD(0, R)$ and the space of unit measures in $\myo\DD(0, R)$ is weakly* compact, there exists some subsequence $\Lambda$ of $\NN$ such that $\lambda_{K_n} \xrightarrow{*}\lambda$, $n\in\Lambda$, where $\lambda\in\Meas_1(\myo\DD(0,R))$. Without lost of generality we put $\Lambda \equiv \NN$. Then it is well known (see~\cite[Lemma 3.3.3]{Ra95}) that
\begin{equation}
\limsup_{n\to\infty} V^{\lambda_{K_n}}\leq V^{\lambda}.
\end{equation}
Moreover (see~\cite[Ch.~I, Theorem 6.9]{SaTo}),
\begin{equation}
\label{envel_t}
\limsup_{n\to\infty} V^{\lambda_{K_n}}=V^{\lambda} \quad \mbox{n. e. in }\CC.
\end{equation}
Since $\gamma_{K_n}=-\log\mcap K_n$, where $\gamma_{K_n}$ is the Robin constant of $K_n$, it follows from~\eqref{envel_t},~\eqref{c_0} and~\eqref{g-V} that
\begin{equation}
\label{g-St}
\limsup_{n\to\infty} g_{\myh\CC\setminus K_n}(z)=V^{\lambda}(z)-\log c_0=:g(z) \quad \mbox{n. e. in }\CC.
\end{equation}
Let
\begin{equation}
\label{zero-F}
F:=\{z: g(z) =0\}.
\end{equation}

Then Stahl proves from~\eqref{g-St} (see~\cite[Lemmas 7--10]{St12}) that there exists a compact set $S\in\KKf$ such that $\mcap S =\mcap \myo F$, where $\myo F$ is the closure of $F$. In order for $S$ to be the desired compact set in Question~\ref{Q1}, we need $\mcap \myo F = c_0$. Stahl deduces this from the statement that $g=0$ nearly everywhere on $\myo F$. But he does a crucial mistake in the proof of this statement.

First of all, let us explain that if $g=0$ nearly everywhere on $\myo F$, then indeed $\mcap \myo F = c_0$. It is clear that if $g=0$ nearly everywhere on $\myo F$,  then $\mcap\myo F=\mcap F$.
If $\mcap F = 0$, we have $\mcap S =\mcap \myo F=0$ that contradicts $c_0>0$. So, let $\mcap F>0$ and $g=0$ nearly everywhere on $\myo F$. Consider the function $h:=g-g_{\myh\CC\setminus\myo F}$. Since by definition~\eqref{g-St} $g$ is subharmonic function in $\CC$ with the logarithmic singularity at $\infty$, 
we have that $h$ is a subharmonic function in the connected component of $\myh\CC\setminus\myo F$ containing $\infty$ and $\limsup\limits_{z\to\xi}h(z)= 0$ for nearly everywhere $\xi\in \myo F$. Then it follows from the extended maximum principle (see~\cite[Theorem 3.6.9]{Ra95}) that $h(z)\leq 0$  for all $z$ from the connected component of $\myh\CC\setminus\myo F$ containing $\infty$. In particular, for $z=\infty$. We have $h(\infty)=-\log c_0 + \log\mcap \myo F$. Thus, $\mcap\myo F\leq c_0$. Since $\mcap S =\mcap \myo F$ and $S\in\KKf$, we claim that $\mcap S =\mcap \myo F=c_0$. Q.E.D. 
Moreover, now we have $h(\infty)=0$, i.e. our subharmonic function $h$ reaches its maximum at an inner point of the domain, therefore $h\equiv 0$ in the connected component of $\myh\CC\setminus\myo F$ containing $\infty$. Thus, $g=g_{\myh\CC\setminus\myo F}$ and our limiting measure $\lambda$ is the equilibrium measure of $\myo F$ (and $S$).

Now let us explain Stahl's mistake in the proof of the statement that $g=0$ nearly everywhere on $\myo F$. Let us say right now that we don't know whether this statement is true or false. We discuss it below.
Stahl states that every potential is continuous nearly everywhere in $\CC$! Of course this is false! For example, see our example~\ref{Ex1} in the end of this section. (Evidently, this wrong statement implies $g=0$ nearly everywhere on $\myo F$.) Moreover, Stahl says that this wrong statement follows from (or is equivalent to) the correct capacitive analogue of classic Luzin's C-property for a logarithmic potential from~\cite[Ch. III, Theorem 3.6]{La66}: for an arbitrary measure $\mu$ and an arbitrary $\varepsilon>0$ there exists an open set $G$ such that $\mcap G<\varepsilon$ and the logarithmic potential $V^\mu$ is continuous on $\CC\setminus G$. Stahl writes this in his original paper~\cite[the last paragraph on p. 316]{St85-1} in 1985 and in the last arXiv paper~\cite[the last paragraph on p. 93]{St12} in 2012. It is surprising that during these 27 years between these two papers nobody had explained to Stahl that he was wrong in his view that C-property and continuity nearly everywhere are equivalent. Moreover, as far as we know, nobody has pointed out this Stahl's mistake before.

In example~\ref{Ex1} we construct such measure $\nu$ with $\supp\nu \subset \myo\DD(0,2)$ that everywhere $V^\nu\geq d$ for some constant $d$, and the capacity of the set $F_0:=\{z: V^\nu(z)=d\}$ is strictly less than the capacity of its closure. Of course, in particular, such potential $V^\nu$ is not continuous nearly everywhere on $\myo{F_0}$ and, consequently, on $\CC$. However, what is more important for us is the following. Consider the function $g_0^:=V^\nu-d$. Just like $g$~\eqref{g-St}, the function $g_0$ is non-negative and $g_0$ is the logarithmic potential of some measure up to a constant. But for $F_0=\{z: g_0(z) =0\}$ we have $\mcap F_0<\mcap\myo{F_0}$ and therefore $g_0$ is not equal to 0 nearly everywhere on $\mcap\myo F$. It is similar to contradiction of Stahl's statement that $g=0$ nearly everywhere on $\myo F$. However, we don't know whether we can satisfy the left-hand side of~\eqref{g-St} for $g_0$, and, moreover, show that $\nu$ is the weak* limit of equilibrium measures $\lambda_n$ of some compact sets $K_n$ that satisfy~\eqref{c_0}.

From the other side we have the following arguments, which show that Stahl's statement that $g=0$ nearly everywhere on $\myo F$ may be true. We see above that if this statement is true, then the limiting measure $\lambda$ is the equilibrium measure of $\myo F$. The opposite is also true. If $\lambda$ is the equilibrium measure of its support, then it immediately follows from Frostman's theorem~\eqref{V^lambda} that $g=0$ nearly everywhere on $\myo F$. (Moreover, it is well known that the logarithmic potential of any equilibrium measure is continuous nearly everywhere in $\CC$.) Thus, if we could prove that $\lambda$ is the equilibrium measure of its support, we would prove $g=0$ nearly everywhere on $\myo F$ and fix the gap in Stahl's proof. So, we have the following open question. 
\begin{Question}
\label{Q2}
Let $f_\infty$ be an arbitrary analytic germ at $\infty$. 
Let $K_n\in \myo\DD(0,R)$ for some $R<\infty$, $K_n\in\KKf$ and 
\begin{equation*}
\lim_{n\to\infty}\mcap K_n = \inf\limits_{K\in\KKf}\mcap K =: c_0,
\end{equation*}
where $c_0>0$.
Let $\lambda_{K_n} \xrightarrow{*}\lambda$ as $n\to\infty$, where $\lambda_{K_n}$ is the equilibrium measure of $K_n$ and $\lambda$ is some measure.
Is it true that $\lambda$ is the equilibrium measure of its support?
\end{Question}

Moreover, we can formulate a more general question without any relation with a germ $f_\infty$.
\begin{Question}
\label{Q3}
Let compact sets $K_n\subset\myo\DD(0,R)$ for some $R<\infty$ and $\mcap K_n\geq c>0$. 
Let $\lambda_{K_n}$ be the equilibrium measure of $K_n$ and $\lambda_{K_n} \xrightarrow{*}\lambda$ as $n\to\infty$, where $\lambda$ is some measure.
Is it true that $\lambda$ is the equilibrium measure of its support?
\end{Question}

The Question~\ref{Q3} seems very natural for the logarithmic potential theory. Nevertheless, we don't know the answer. Evidently, the positive answer to Question~\ref{Q3} implies the positive answer to Question~\ref{Q2}, and, as we explain above, the positive answer to Question~\ref{Q2} implies the positive answer to Question~\ref{Q1}. But, unfortunately, we don't know the answer to any of these questions.

\begin{Example}
\label{Ex1}
Let $\{r_n\}$, $n\in\NN$ be a sequence of all points in (closed) unit disc $\myo\DD(0,1)$ with rational coordinates. Let us consider the measure 
\begin{equation}
\label{ex1_mu}
\mu := \dfrac{6}{\pi^2}\sum_{n=1}^\infty \delta_{r_n}. 
\end{equation}
Evidently, $\supp \mu =\myo\DD(0,1)$ and $\mu(\myo\DD(0,1))=1$. Since $V^\mu$ is a subharmonic function, its exceptional set $F_\mu:=\{z:V^\mu(z)=-\infty\}$ is polar, i.e. $\mcap F_\mu=0$. But $F_\mu$ contains all the points $r_n$, therefore $\myo{F_\mu}=\myo\DD(0,1)$ and $\mcap \myo{F_\mu}=1$. It is a good example but we would like to replace $-\infty$ with some finite number. For this goal we consider a truncation of $V^\mu$.

For each constant $c$ we define the set
\begin{equation*}
J_c:= \{z\in \myo\DD(0,1): V^\mu(z)\geq c\}. 
\end{equation*}
Since each subharmonic function (in paricular, $V^\mu$) is upper continuous, the sets $J_c$ are compact.
We have
\begin{equation*}
\myo\DD(0,1)\setminus F_\mu = \bigcup\limits_{-\infty<c\leq t_0} J_c \quad\mbox{and } J_{t_2}\supset J_{t_2} \mbox{ for } t_2<t_1,   
\end{equation*}
i.e. the set $\myo\DD(0,1)\setminus F_\mu$ is the limit of an increasing sequence of Borel sets. Therefore (see~\cite[Theorem 5.1.3]{Ra95}),
\begin{equation}
\label{lim_J}
\mcap (\myo\DD(0,1)\setminus F_\mu) = \lim_{t\to-\infty}\mcap J_t.   
\end{equation}
Since $\mcap F_\mu = 0$, we have $\mcap (\myo\DD(0,1)\setminus F_\mu)=1$ and it follows from~\eqref{lim_J} that $\lim_{t\to-\infty} \mcap J_t =1$.
Let us choose such $d<0$ that $\mcap J_{d/2}> 1/2$ and consider the function $\max (d, V^\mu(z))$. It is a subharmonic function in $\CC$ (as the maximum of two subharmonic functions) and it is $\log|z|+o(1)$ as $z\to\infty$. Therefore this function is the logarithmic potential of the measure $\nu:=\ddc \max (d, V^\mu(z))$, where $\ddc$ is the generalized Laplace operator. So,
\begin{equation}
\label{Ex1_V^nu}
V^{\nu}(z)=\max (d, V^\mu(z)).
\end{equation}
Put $F_\nu:=\{z:V^\nu(z)=d\}$.
The set $F_\nu$ contains all points at that $V^\mu= -\infty$, in particular, all the points of $\myo\DD(0,1)$ with rational coordinates. Therefore $\myo{F_\nu}\cap \myo\DD(0,1)=\myo\DD(0,1)$. But $J_{d/2}\subset \myo\DD(0,1)$, $\mcap J_{d/2}>1/2$ and $V^\nu>d$ on $J_{d/2}$.  

Thus we construct a measure $\nu$ such that $V^\nu\geq d$ on whole $\CC$ and $V^\nu$ is not equal $d$ nearly everywhere on the closure of the set $F_\nu=\{z:V^\nu(z)=d\}$. Moreover, it is clear from our construction that for each $\varepsilon<1$ we can construct a measure $\nu$ such that the set $\{z\in\myo F_\nu: V^\nu(z)>d\}$ contains a set $E$ of capacity greater then $\varepsilon$. In our case $\varepsilon = 1/2$ and $E=J_{d/2}$.

Notice that $\supp \nu \subset \myo\DD(0,2)$. Indeed, it is clear from the definition~\eqref{ex1_mu} that $V^\mu>0$ outside $\myo\DD(0,2)$ and $V^\mu$ is harmonic outside $\myo\DD(0,1)$. Consequently, it follows from~\eqref{Ex1_V^nu} and our choice $d<0$ that $\supp \nu \subset \myo\DD(0,2)$.

\end{Example}

\section{Preliminary results}
\label{s3}

In this section we present preliminary results from the general topological theory and from the logarithmic potential theory. Many results from the former part are contained in~\cite{St12} (sometimes without proofs) and many results from the latter part are contained in~\cite{PeRa}. Nevertheless we present these results with full proofs for the completeness of the presentation.

We denote the (open) disc with a centre at $x$ and of radius $R$ by $\DD(x, R):=\{z\in\CC: |z-x|<R\}$. 

We denote the circle with a centre at $x$ and of radius $R$ by $\TT(x, R):=\{z\in\CC: |z-x|=R\}$.

For any set $G$ we denote its closure by $\myo G$.  

For any set $G$ we denote its boundary by $\partial G$.  

We denote the distance in Euclidean metric by $\dist(\cdot,\cdot)$.

We denote the (open) $\varepsilon$-neighbourhood in Euclidean metric of any set $G\subset\CC$ by $G^\varepsilon$, i.e. 
\begin{equation*}
G^{\varepsilon}:=\{z: \dist(z, G)<\varepsilon\},
\end{equation*}

For an arbitrary compact set $F\in\CC$ we denote the unbounded component of its complement by $F_\infty$.

For an arbitrary compact set $F$ we denote its polynomially convex hull by $\myh F$, i.e. $\myh F:=\myh\CC\setminus F_\infty$.

We begin with a simple and well-known Lemma~\ref{L0}, but we don't know the direct link on it.

\begin{Lemma}
\label{L0}
For any connected compact set $K\subset \CC$ with connected complement and any $\varepsilon>0$ there exists a simply connected neighbourhood $U$ of $K$ such that $U\subset K^\varepsilon$ and $\partial U$ is a closed Jordan curve. 
\end{Lemma}

\begin{proof}
If $K$ is a point, the statement is trivial. Let $K$ be not a point. 
By Riemann mapping theorem, there exists a conformal map $\varphi$ of $\myh\CC\setminus K$ onto the unit disc $\DD(0,1)$ that maps $\infty$ to~$0$. For each $r\in(0,1)$ we put
\begin{equation*}
\rho(r):=\max\limits_{|\xi| = r}\dist(\varphi^{-1}(\xi), K).
\end{equation*}
Let us show that $\rho(r)\to 0$ as $r\to 1$. Let us assume the opposite. Then there exist a sequence of numbers $r_n\in(0,1)$, $n\in\NN$, a sequence of points $\xi_n\in\DD(0, R)$, $n\in\NN$, and $\delta>0$ such that $r_n\to 1$, $|x_n|=r$ and $\dist(\varphi^{-1}(x_n), K)>\delta$. Choosing subsequences we assume that $\varphi^{-1}(x_n)$ converges to some point $z_0\in\CC\setminus K$. Since $\varphi$ is continuous, in particular, at $z_0$, we have $x_n\to \varphi(z_0)$ as $n\to\infty$. But $|x_n|\to 1$ and $|\varphi(z_0)|<1$. Contradiction.

Let us choose $r$ such that $\rho(r)<\varepsilon$. Let us show that $U:=\myh\CC\setminus\varphi^{-1}(\myo\DD(0,r))$ is the desired neighbourhood. Indeed, $U\subset K^\varepsilon$ by construction. Moreover, since the image of a closed Jordan curve under a conformal map is also a closed Jordan curve, we have that $\partial U$ is a closed Jordan curve.
\end{proof}

The following Lemma~\ref{L1} is some technical statement from the theory of analytic continuation in the complex plane.

\begin{Lemma}
\label{L1}
Let $f_\infty$ be an analytic germ at $\infty$. Let $\gamma_1$ and $\gamma_2$ be paths from $\infty$ to some point $z_0\in\CC$ such that there exist analytic continuations $f_{z_0}^1$ and $f_{z_0}^2$ of $f_\infty$ along $\gamma_1$ and $\gamma_2$ to $z_0$ respectively and $f_{z_0}^1\ne f_{z_0}^2$. Then there exist a point $\myt z_0\in\CC$ and Jordan non-intersecting at inner points paths $\myt\gamma_1$ and $\myt\gamma_2$ from $\infty$ to $\myt z_0$ such that there exist analytic continuations $f_{\myt z_0}^1$ and $f_{\myt z_0}^2$ of $f_\infty$ along $\myt\gamma_1$ and $\myt\gamma_2$ to $\myt z_0$ respectively and $f_{\myt z_0}^1\ne f_{\myt z_0}^2$. 
\end{Lemma}

\begin{proof}
It immediately follows from the theorem on analytic continuation along close paths that there exist piecewise linear paths $\gamma^\circ_1$ and $\gamma^\circ_2$ such that
there exist analytic continuations of $f_\infty$ to $z_0$ along paths $\gamma^\circ_1$ and $\gamma^\circ_2$, which give $f_{z_0}^1$ and $f_{z_0}^2$, respectively.
Moreover, it is clear that we can choose $\gamma^\circ_1$ and $\gamma^\circ_2$  such that each of them self-intersects at finite number of points and $\gamma^\circ_1$ intersects $\gamma^\circ_2$ also at finite number of points. For simplicity of notation, further we will denote these $\gamma^\circ_1, \gamma^\circ_2$ by$\gamma_1, \gamma_2$.

Let $\gamma_1: [0,1]\to\CC$ self-intersect. Let $\gamma_1(t_2)$ be the smallest (with respect to parameter $t$ of $\gamma_1(t)$) self-intersection point of $\gamma_1$, i.e. $\gamma_1(t_1) = \gamma_1(t_2)$ for some $t_1<t_2$ and $\gamma_1(s_1) \ne \gamma_1(s_2)$ for all $s_1,s_2<t_2$.
Note that $\gamma_1(t), t\in[t_1,t_2]$, is a closed Jordan path.
Let the analytic continuations of $f_\infty$ along $\gamma_1(t), t\in[0,t_1]$, and along $\gamma_1(t), t\in[0,t_2]$, give two different germs at the point $\gamma_1(t_1)=\gamma_1(t_2)$. Then it is clear that we can choose paths $\myt\gamma_1, \myt\gamma_2$  from $\infty$ to $\gamma_1(t_1)$ close to $\gamma_1(t), t\in[0,t_1]$, and $\gamma_1(t), t\in[0,t_2]$, respectively, such that $\myt\gamma_1, \myt\gamma_2$ intersect each other only at their start and end points and the analytic continuation of $f_\infty$ along $\myt\gamma_1$ coincides with the analytic continuation along $\gamma_1(t), t\in[0,t_1]$, and the analytic continuation of $f_\infty$ along $\myt\gamma_2$ coincides with the analytic continuation along $\gamma_1(t), t\in[0,t_2]$. Thus, in this case $\myt\gamma_1, \myt\gamma_2$ are desired paths.
Let now the analytic continuations of $f_\infty$ along $\gamma_1(t), t\in[0,t_1]$, and along $\gamma_1(t), t\in[0,t_2]$, give the same germ at the point $\gamma_1(t_1)=\gamma_1(t_2)$. Then we can cut the loop $\gamma_1(t), t\in[t_1, t_2]$. Indeed, the path $\gamma_1(t), t\in[0,t_1]\cup [t_2,1]$, has strictly fewer self-intersection points than $\gamma_1(t), t\in[0,1]$, and the analytic continuation along both these paths give the same germ $f_{z_0}^1$. Now we can repeat all this procedure for the path $\gamma_1(t), t\in[0,t_1]\cup [t_2,1]$, an so on. As a result we get either desired paths $\myt\gamma_1, \myt\gamma_2$ or a Jordan path $\breve\gamma_1$ from $\infty$ to $z_0$ such that there exists the analytic continuation of $f_\infty$ along $\breve\gamma_1$ that gives the germ $f_{z_0}^1$.

Now we do the same procedure for $\gamma_2$. As a result we have either desired paths $\myt\gamma_1, \myt\gamma_2$ or two Jordan paths $\breve\gamma_1, \breve\gamma_2$ from $\infty$ to $z_0$ such that $\breve\gamma_1, \breve\gamma_2$ intersect each other at finite numbers of points and there exist the analytic continuations of $f_\infty$ along $\breve\gamma_1$ and $\breve\gamma_2$ that give the germs $f_{z_0}^1$ and $f_{z_0}^2$, respectively.
Let us consider the second situation. Let $\breve\gamma_1: [0,1]\to\myh\CC$ and $\breve\gamma_2: [0,1]\to\myh\CC$. Let $t_1, s_1\in(0,1]$ be such that $\breve\gamma_1(t_1)=\breve\gamma_2(s_1)$ and $\breve\gamma_1(t)\ne\breve\gamma_2(s)$ for all $t<t_1$, $s\in(0,1]$ and $\breve\gamma_1(t_1)\ne\breve\gamma_2(s)$ for all $s<s_1$. 
Note that $\breve\gamma_1(t), t\in[0,t_1]$, and $\breve\gamma_2(t), t\in[0,s_1]$, are two Jordan paths from $\infty$ to $\breve\gamma_1(t_1)=\breve\gamma_2(s_1)$ that don't intersect each other at inner points. Therefore if the analytic continuations of $f_\infty$ along $\breve\gamma_1(t), t\in[0,t_1]$, and $\breve\gamma_2(t), t\in[0,s_1]$, give two different germs, these paths are desired paths $\myt\gamma_1, \myt\gamma_2$. 
Let us suppose that the analytic continuations of $f_\infty$ along $\breve\gamma_1(t), t\in[0,t_1]$, and $\breve\gamma_2(t), t\in[0,s_1]$, give the same germ. Then the analytic continuation of $f_\infty$ along $\breve\gamma_2(t), t\in[0,1]$, coincides with the analytic continuation along the path $\breve\gamma_3$ that is the union of $\breve\gamma_1(t), t\in[0,t_1]$, and $\breve\gamma_2(t), t\in[s_1,1]$. It is clear that we can choose a path $\breve\gamma_4$ from $\infty$ to $z_0$ close to $\breve\gamma_3$ such that the analytic continuations along them coincide, $\breve\gamma_4$ is the union of some path $\breve\gamma$ from $\infty$ to $\breve\gamma_2(s_0)$ for some $s_0>s_1$ and $\breve\gamma_2(t), t\in[s_0,1]$, and the paths $\breve\gamma$ and $\breve\gamma_1$ don't intersect each other. So, the analytic continuation of $f_\infty$ along $\breve\gamma_4$ also as along $\breve\gamma_2$ gives the germ $f_{z_0}^2$, but the number of intersection points of $\breve\gamma_4$ and $\breve\gamma_1$ is strictly fewer than of $\breve\gamma_2$ and $\breve\gamma_1$. Repeating this procedure several times we get the desired paths $\myt\gamma_1, \myt\gamma_2$.
\end{proof}

The following topological Lemma~\ref{L2} is~\cite[Lemma 3]{St12}. But Stahl's proof of it is incorrect. (In particular, Stahl uses the wrong statement that any compact set has countable set of connected components.) We give a correct simple proof of this lemma. The idea of the proof was kindly communicated to us by E. M. Chirka.

\begin{Lemma}
\label{L2}
Let $\RRR$ be a ring domain in $\myh\CC$, i.e. $\myh\CC\setminus \RRR$ is a union of two (non-intersecting) closed components $A_1, A_2$. Let $K$ be a compact set in $\myh\CC$ such that $K\cap\gamma\ne\varnothing$ for each closed Jordan path $\gamma\in \RRR$ that separates $A_1$ from $A_2$. Then there exists a connected component $V$ of $K$ such that $V\cap A_1\ne\varnothing$ and $V\cap A_2\ne\varnothing$.
\end{Lemma}

\begin{proof}
Note that $K\cap A_1\ne\varnothing$ and $K\cap A_2\ne\varnothing$. Indeed, if $K\cap A_1 = \varnothing$, then there exists some $\varepsilon$-neighbourhood $A_1^{\varepsilon}$ of $A_1$ such that $K\cap A_1^{\varepsilon}=\varnothing$ and $A_2\cap A_1^{\varepsilon}=\varnothing$. Let $B$ be a finite subcover of $\varepsilon/2$-cover of $A_1$. Then the boundary of the connected component of $\myh\CC\setminus B$ that contains $A_2$ is a Jordan path that belongs to $\RRR$ and doesn't intersect $K$. Contradiction.

Let $F$ be the connected component of $A_1$ in the compact set $K\cup A_1$. If $F\cap A_2\ne\varnothing$, then the connected component of any point $z\in F\cap A_2$ in $K$ is the desired connected component $V$. Let us suppose that $F\cap A_2=\varnothing$. Then there exists a neighbourhood $U$ of $F$ that doesn't intersect $A_2$. Then (see~\cite[Corollary after Theorem~21]{Al77}) there exists a neighbourhood $\myt U$ of $F$ such that $\myt U\subset U$ and $\partial\myt U$ doesn't intersect $K\cup A_1$, where $\partial\myt U$ is the boundary of $\myt U$. Consequently, some $\delta$-neighbourhood of $\partial\myt U$ doesn't intersect $K\cup A_1$. Let $C$ be a finite subcover of $\delta/2$-cover of $\partial\myt U$. Then the boundary of the connected component of $\myh\CC\setminus C$ that contains $A_2$ is a Jordan path that belongs to $\RRR$ and doesn't intersect $K$. Contradiction.
\end{proof}

In~\cite[Lemma 5]{St12} Stahl states that when we consider a minimizing sequence of compact sets $K_n$ in Question~\ref{Q1}, we can suppose that all $K_n$ belong to some big disc $\myo\DD(0,r)$. For a proof of this he states that if $K\in\KKf$, then the radial projection of $K$ on some big disc $\myo\DD(0,r)$ also belongs to $\KKf$. But he doesn't prove this statement. We don't know whether this statement is true or not. Therefore we prove the following Proposition~\ref{Prop_D}, from which, evidently, it follows that we can choose $K_n$ from some big disc. Recall that $\mcap \myo\DD(0, R)=R$. Thus, if $f_\infty\in\MM(\myh\CC\setminus\myo\DD(0,R))$, then the infimum $c_0$~\eqref{c_0} in Question~\ref{Q1} is not greater than $R$.

\begin{Proposition}
\label{Prop_D}
Let $f_\infty$ be a germ at $\infty$ and $f_\infty\in\MM(\myh\CC\setminus \myo\DD(0,R))$ for some $R>0$. Let $K\in\KKf$ and $\mcap K\leq R$. Then there exists a compact set $F\subset K$ such that $F\in\KKf$ and $F\subset\DD(0,6R)$.
\end{Proposition}

\begin{proof}
Let $V$ be the union of all connected components of $K$ that intersect $\myo\DD(0, R)$. In other words, $V$ is the intersection of $K$ and the connected component of $\myo\DD(0, R)$ in the compact set $K\cup \myo\DD(0, R)$. Therefore $V$ is a compact set. Let us show that $V\subset\myo\DD(0, 5R)$. Indeed, if $z\in V$ and $|z|>5R$, then the connected component $K_z$ of $z$ in $K$ intersects $\myo\DD(0, R)$ and, consequently, has diameter not less than $|z|-R>4R$. Therefore $\mcap K \geq\mcap K_z>R$, because the capacity of any connected compact set is not less than a quarter of its diameter. Contradiction.

So, $V\subset\myo\DD(0, 5R)$. Therefore $V\cup\myo\DD(0, R)\subset\DD(0,6R)$ and $V\cup\myo\DD(0, R)$ is a connected component in $K\cup\myo\DD(0, R)$ by definition. Consequently, see~\cite[Corollary after Theorem~21]{Al77}, there exists a neighbourhood $U$ of $V\cup\myo\DD(0, R)$ such that $U\subset\DD(0,6R)$ and the boundary $\partial U$ of $U$ doesn't intersect $K\cup\myo\DD(0, R)$. 
Then the boundary $\partial\myh U$ of $\myh U$ is a subset of $\partial U$ and, consequently, $\partial \myh U$ doesn't intersect $K\cup\myo\DD(0, R)$. Therefore $F:= K\cap\myh U$ and $K\setminus F = K\setminus \myh U$ are compact sets. Let us show that $F$ is the desired compact set.

$F\subset\DD(0,6R)$ by definition. Since $\myh\CC\setminus K$ is connected and $F$, $K\setminus F$ are compact sets, we have that $\myh\CC\setminus F$ and $\myh\CC\setminus (K\setminus F)$ are also connected.
Let us show that $F\in\KKf$, i. e. $f_\infty\in\MM(\myh\CC\setminus F)$. 
Since $\partial\myh U\cap K=\varnothing$, there exists an $\varepsilon$-neighbourhood $(\partial\myh U)^\varepsilon$ of $\partial\myh U$ such that $(\partial\myh U)^\varepsilon\cap K =\varnothing$. Applying Lemma~\ref{L0} we obtain a neighbourhood $W$ such that $\myo{\myh U}\subset W\subset U^\varepsilon$, the boundary $\partial W$ is a closed Jordan curve and  $\partial W$ separates $F$ and $K\setminus F$.
We have $f_\infty\in\MM(\myh\CC\setminus\myo\DD(0, R))$ by the condition and we denote this meromorphic function also by $f_\infty(z)$.
Note that $\myo{W_\infty}=\myh\CC\setminus W \Subset \myo{\myh U_\infty}=\myh\CC\setminus \myh U\Subset \myh\CC\setminus\DD(0, R)$ by construction. 
Since $\myh\CC\setminus K$ is connected, there exist a point $z_*\in\partial W$ and a path $\gamma_*$ from $\infty$ to $z_*$ such that $\gamma_*\subset\myo{W_\infty}$ and $\gamma_*\cap K=\varnothing$. Since $\partial W\cap K=\varnothing$, we now can continue moving from $z_*$ along $\partial W$ and reach any point $z\in\partial W$. As a result, for any $z\in\partial W$ we obtain a path $\gamma_z$ from $\infty$ to $z$ such that $\gamma_z\subset\myo{W_\infty}$ and $\gamma_z\cap K=\varnothing$.
Consequently, the analytic continuation of $f_\infty$ along $\gamma_z$ gives $f_\infty(z)$. We will use such $\gamma_z$ further.

Let $\gamma:[0, 1] \to \myh\CC\setminus F$ be a closed path from $\infty$ to $\infty$. We need to show that there exists the analytic continuation of $f_\infty$ along $\gamma$ and this continuation gives $f_\infty$ again. 
Let $t_0$ be supremum of $t\in [0,1]$ such that there exists the analytic continuation of $f_\infty$ along $\gamma(s), s\in[0, t]$.
For each $t\in[0, t_0)$ we denote the germ that is the analytic continuation of $f_\infty$ along $\gamma(s), s\in[0, t]$, to the point $\gamma(t)$ by $f_{\gamma(t)}$.
Let us put
\begin{equation*}
t_+:=\sup\{t\in[0,t_0): \gamma(t)\in \myo{W_\infty} \mbox{ and } f_{\gamma(t)}(z)=f_\infty(z)\}.
\end{equation*}
By definition $t_+>0$. Let us assume that $t_+=t_0$. Since $\gamma(t_0)\in \myo{W_\infty}$ and $f_\infty\in\MM(\myo{W_\infty})$, we have that there exists the analytic continuation along $\gamma(t)$, $t\in[0,t_0]$, and gives $f_\infty(z)$. Consequently, if $t_0=1$ we have shown that there exists the analytic continuation of $f_\infty$ along $\gamma$ and it gives $f_\infty$ and if $t_0<1$ we have obtained a contradiction to the definition of $t_0$.
Let now assume that $t_+<t_0$. Let us assume that there is no $t\in [0,t_0)$ such that $\gamma(t)\in \myo{W_\infty}$ and $f_{\gamma(t)}(z)\ne f_\infty(z)$.
Then the analytic continuations of $f_\infty$ to $\gamma(t_0)$ along $\gamma(t), t\in[0, t_0],$ and along $\gamma_{t_+}\cup \gamma(t), t\in[t_+,t_0]$ coincide.
But the path $\gamma_{t_+}\cup \gamma(t), t\in[t_+,t_0]$ belongs to $\myh\CC\setminus K$, because $K$ and $F$ coincide outside $\myo{W_\infty}$. Consequently, there exists the analytic continuation along $\gamma(t), t\in[0, t_0],$ contradiction. So, we can correctly define
\begin{equation*}
t_-:=\inf\{t\in[0,t_0): \gamma(t)\in \myo{W_\infty} \mbox{ and } f_{\gamma(t)}(z)\ne f_\infty(z)\}.
\end{equation*}
Reasoning in the same way as in the proof of correctness of the definition of $t_-$ we obtain $t_+<t_-$. Let us consider the path $\myt\gamma:=\gamma_{t_+}\cup\gamma(t), t\in[t_+,t_-]\cup-\gamma_{t_-}$, where $-\gamma_{t_-}$ is the path $\gamma_{t_-}$ taken in the other direction. The path $\myt\gamma\subset\myh\CC\setminus K$, but the analytic continuation along it gives a germ that is different from $f_\infty$. Contradiction.
\end{proof}

The following corollary immediately follows from Proposition~\ref{Prop_D}.

\begin{Corollary}
\label{Cor_D}
Let $f_\infty$ be a germ at $\infty$ and $f_\infty\in\MM(\myh\CC\setminus \myo\DD(0,R))$ for some $R>0$. Then we can choose a minimizing sequence of compact sets $K_n$~\eqref{c_0} in Question~\ref{Q1} (or in Theorem~\ref{t1}) such that all $K_n$ belong to $\DD(0,6R)$.
\end{Corollary}

The following Statement~\ref{St_c_0} is exactly~\cite[Lemma 4]{St12}. Unfortunately, Stahl's proof of it is incorrect, because he uses~\cite[Lemma 2]{St12}, but gives obviously wrong proof of this lemma. We don't know whether~\cite[Lemma 2]{St12} is true or not.

\begin{Statement}
\label{St_c_0}
Let $f_\infty$ be an analytic germ at $\infty$ and
\begin{equation*}
c_0:= \inf\limits_{K\in\KKf}\mcap K=0.
\end{equation*}
Then $f_\infty\in\MM(\myh\CC\setminus E)$, where $E$ is a polar compact set.
\end{Statement}

\begin{proof}
Let us suppose that $f_\infty$ has multivalued continuation at some point $z_0$. It means that there exist two paths $\gamma_1,\gamma_2$ from $\infty$ to $z_0$ such that there exist analytic continuations of $f_\infty$ along these paths and we get two different germs $f^1_{z_0}, f^2_{z_0}$, respectively, at $z_0$. Using Lemma~\ref{L1}, we suppose that $\gamma_1,\gamma_2$ are Jordan paths that don't intersect each other in inner points. 
Since $f_\infty$ is analytic at $\infty$, there exists such $R>0$ that $f_\infty\in\MM(\myh\CC\setminus\myo\DD(0,R))$. Therefore we can slightly change the paths $\gamma_1,\gamma_2$ in such a way that they start from some point $z_*$ in a small neighbourhood of infinity and analytic continuations of $f_*$ along $\gamma_1, \gamma_2$ give two different germs at $z_0$, where $f_*$ is the analytic continuation of $f_\infty$ in $z_*$ along any path in $\myh\CC\setminus\myo\DD(0,R)$.
We suppose that $z_*\in\myh\CC\setminus\myo\DD(0,10R)$. 
Due to the theorem on analytic continuation along close paths we suppose that $\gamma_1,\gamma_2$ are piecewise linear paths.
In particular, the path $\gamma =\gamma_1\cup-\gamma_2$ is a closed piecewise linear Jordan path, where $-\gamma_2$ is the path $\gamma_2$ taken in the other direction.
Moreover, we also can slightly change the paths $\gamma_1,\gamma_2$ such that $z_*$ belongs to only one line segment $I_*$ of $\gamma$ and $z_0$ belongs to only  one line segment $I_0$ of $\gamma$ (i.e. $z_*, z_0$ are not vertices of $\gamma$) and the angles between $I_*$ and its adjacent segments and between $I_0$ and its adjacent segments are greater than $\pi/4$. See the picture.

Recall that we denote (open) $\varepsilon$-neighbourhood in Euclidean metric of any set $G\subset\CC$ by $G^\varepsilon$. 
Let $d_0$ be the minimum of the distances between $z_0$ and the endpoints of $I_0$ and $d_*$ between $z_*$ and the endpoints of $I_*$. Let $d=\min(d_0, d_*)$. 
Let $\varepsilon$ be such small that 1) $4\varepsilon$ is less than the minimum of the distances between any non-adjacent segments of $\gamma$, 2) $3\varepsilon<d$, 3) $f_*\in\MM(\gamma_1^{3\varepsilon})$ and $f_*\in\MM(\gamma_2^{3\varepsilon})$, 4) $\varepsilon<R$.
Let $K\in\KKf$ and $K\in\DD(0, 6R)$. We will show that $\mcap K\geq\varepsilon/2$. Then it will immediately follow from  Corollary~\ref{Cor_D} that $c_0\geq\varepsilon/2$. Contradiction with the assumption $c_0=0$.

So, let us consider the ring domain $\gamma^\varepsilon$, see the picture. The boundary of $\gamma^\varepsilon$ is a union of two closed Jordan curves $Q_1, Q_2$, which are separated by $\gamma$. Let $D_1^\varepsilon:=\gamma_1^{2\varepsilon}\cap\gamma^\varepsilon$ (red domain on the picture) and $D_2^\varepsilon:=\gamma_2^{2\varepsilon}\cap\gamma^\varepsilon$ (green domain on the picture). Let $B_0^1$ and $B_*^1$ be the two components of the boundary of $\gamma_1^{2\varepsilon}$ in $\gamma^\varepsilon$ in neighbourhoods of $z_0$ and $z_*$, respectively. Let $B_0^2$ and $B_*^2$ be the two components of the boundary of $\gamma_2^{2\varepsilon}$ in $\gamma^\varepsilon$ in neighbourhoods of $z_0$ and $z_*$, respectively.
Thus, $B_0^1, B_0^2$ are two arcs of angle $\pi/3$ of the circle $\TT(z_0,2\varepsilon)$ and $B_*^1, B_*^2$ are two arcs of angle $\pi/3$ of the circle $\TT(z_*,2\varepsilon)$. Let $G_0:=\gamma^\varepsilon\cap\DD(z_0, 2\varepsilon)$ (red-green domain around $z_0$) and $G_*:=\gamma^\varepsilon\cap\DD(z_*, 2\varepsilon)$ (red-green domain around $z_*$). Then $D_1^\varepsilon\cap D_2^\varepsilon$ = $G_0\cup G_*$.

\begin{center}
\includegraphics[width=16cm,height=9cm]{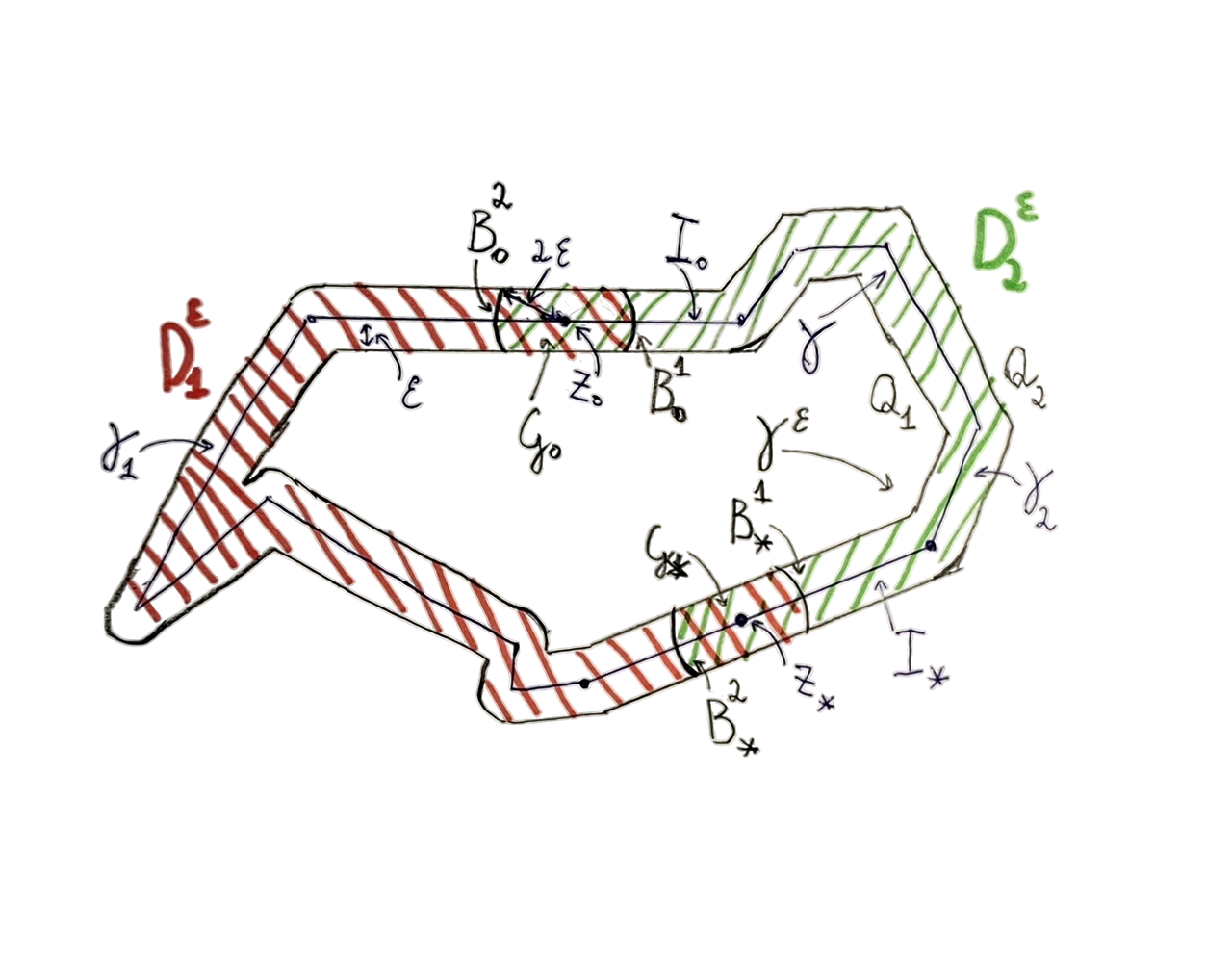}
\end{center}

Let us apply Lemma~\ref{L2} for $\RRR=\gamma^\varepsilon$ and the compact set $K$. We have two possibilities, either there exists a connected component $V$ of $K$ such that $V\cap Q_1\ne\varnothing$ and $V\cap Q_2\ne\varnothing$, or there exists such closed Jordan path $\alpha$ separating $Q_1$ from $Q_2$ that doesn't intersect $K$. Let us suppose that the first holds true. Then the diameter of $V\geq 2\varepsilon$, because the distance between $Q_1, Q_2$ is $2\varepsilon$. Since the capacity of any connected compact set is not less then a quarter of its diameter (see~\cite[Theorem 5.3.2]{Ra95}), we have $\mcap K\geq\mcap V\geq \varepsilon/2$.

Let us now suppose that the second holds true. So, there exists such $\alpha$. Then there exists an arc $\alpha_1$ of $\alpha$ such that $\alpha_1\subset \myo{D_1^\varepsilon}$ and the initial point of $\alpha_1$ belongs to $B_*^1$ and the final point of $\alpha_1$ belongs to $B_0^1$.
It holds true, because otherwise $\alpha$ doesn't separate $Q_1$ from $Q_2$. Analogously, there exists an arc $\alpha_2$ of $\alpha$ such that $\alpha_2\subset \myo{D_1^\varepsilon}$ and the initial point of $\alpha_2$ belongs to $B_*^2$ and the final point of $\alpha_2$ belongs to $B_0^2$.
Let us connect $z_*$ with the initial point of the path $\alpha_1$ by some curve in $\myo{G_*}$ and get a path $\myt\alpha_1$ from $z_*$ to some point on $B_0^1$ that belongs to $\myo{D_1^\varepsilon}$. Since $G_*\subset\myh\CC\setminus\DD(0,8R)$, we have $K\cap \myo{G_*}=\varnothing$. Therefore $\myt\alpha_1\cap K =\varnothing$ and there exists the analytic continuation of $f_*$ along $\myt\alpha_1$. 
Analogously, we construct a path $\myt\alpha_2$ from $z_*$ to some point on $B_0^2$ that belongs to $\myo{D_2^\varepsilon}$, $\myt\alpha_2\cap K =\varnothing$ and there exists the analytic continuation of $f_*$ along $\myt\alpha_2$. But the analytic continuation of $f_*$ along $\myt\alpha_1$ to any point of $G_0$ gives the germ that is the immediate analytic continuation of $f^1_{z_0}$ and along $\myt\alpha_2$ to any point of $G_0$ gives the germ that is the immediate analytic continuation of $f^2_{z_0}$. Consequently, since $f^1_{z_0}\ne f^2_{z_0}$ and $K\in\KKf$, the compact set $K$ should separate parts of $\myt\alpha_1$ that belong to $G_0$ from parts of $\myt\alpha_2$ that belong to $G_0$. By construction there exist such arc $\breve\alpha_1$ of $\myt\alpha_1$ that $\breve\alpha_1\subset \myo{G_0}$ and $\breve\alpha_1\cap B_0^1\ne\varnothing$, $\breve\alpha_1\cap B_0^2\ne\varnothing$. Analogously, there exists such arc $\breve\alpha_2$ of $\myt\alpha_2$ that $\breve\alpha_2\subset \myo{G_0}$ and $\breve\alpha_2\cap B_0^1\ne\varnothing$, $\breve\alpha_2\cap B_0^2\ne\varnothing$. Consequently, $K$ separates $\breve\alpha_1$ from $\breve\alpha_2$ in $G_0$. In particular, the projection of $K$ on $I_0$ should contain the intersection of the projections of $\breve\alpha_1$ and $\breve\alpha_2$ on $I_0$. It is clear that this intersection always contains the interval $I$ of the length $2\sqrt{3}\varepsilon$ that is the difference between the projections of $G_0$ and $B_0^1\cup B_0^2$ on $I_0$. Since the capacity of the projection of any compact set on any line is not greater than the capacity of the set itself (see~\cite[Theorem 5.3.1]{Ra95}), we have $\mcap K\geq\sqrt{3}\varepsilon/2$.  
\end{proof}

Now we pass to the second part of this section related to the logarithmic potential theory.
Recall that we denote (open) $\varepsilon$-neighbourhood in Euclidean metric of any set $G\subset\CC$ by $G^\varepsilon$, i.e. 
\begin{equation*}
G^{\varepsilon}:=\{z: \dist(z, G)<\varepsilon\}.
\end{equation*}
For any non-polar compact set $K$ we denote the level line of Green function of its complement with the singularity at $\infty$ at the level $\log(1+\varepsilon)$ by $(K)_\varepsilon$, i.e.
\begin{equation*}
(K)_\varepsilon:= \{z: g_{\myh\CC\setminus K}(z) = \log(1+\varepsilon)\}.
\end{equation*}
It is clear that
\begin{equation*}
\myh{(K)_\varepsilon} = \{z: g_{\myh\CC\setminus K}(z) \leq \log(1+\varepsilon)\},
\end{equation*}
i.e. $\myh{(K)_\varepsilon}$ is the compact set bounded by $(K)_\varepsilon$.

The following Lemma~\ref{L3} is exactly~\cite[Lemma~1]{PeRa}, see also~\cite[Section~3]{AptYa}. We give the proof of this lemma, which in fact repeats the original proof from~\cite{PeRa}, for the completeness of the presentation.

\begin{Lemma}
\label{L3}
Let $\Gamma$ be a connected compact set that isn't a point. Then 
\begin{equation*}
\dist(\Gamma, (\Gamma)_\varepsilon)\geq\frac{\varepsilon^2}{(1+\varepsilon)}\mcap \Gamma 
\end{equation*}
for any $\varepsilon>0$.
In particular, for $\varepsilon\in(0,1)$ we have  $\dist(\Gamma, (\Gamma)_\varepsilon)>\frac{\varepsilon^2}{2}\mcap \Gamma$.
\end{Lemma}
\begin{proof}

Since $\dist(\Gamma, (\Gamma)_\varepsilon)\geq\dist(\myh\Gamma, (\Gamma)_\varepsilon)$ and $g_{\myh\CC\setminus\Gamma} = g_{\myh\CC\setminus\myh\Gamma}$, we assume that $\myh\CC\setminus\Gamma$ is connected.
It is well known that $g_{\myh\CC\setminus \Gamma}(z)=\log|\varphi(z)|$, where $\varphi$ is the conformal map of $\myh\CC\setminus\myh\Gamma$, i.e. the unbounded component of $\myh\CC\setminus \Gamma$, onto $\myh\CC\setminus\myo\DD(0,1)$ such that $\varphi(z)=(\mcap \Gamma)^{-1} z + O(1)$ as $z\to\infty$. The existence of $\varphi$ follows from Riemann mapping theorem. We denote the inverse function of $\varphi(z)$ by $\Phi(\xi)$. Then $\Phi$ is the conformal map of $\myh\CC\setminus\myo\DD(0,1)$ onto $\myh\CC\setminus \myh\Gamma$ such that $\Phi(\xi)=(\mcap \Gamma) \xi+O(1)$ as $\xi\to\infty$. From the distortion theorem~\cite[Chapter 4, \S2, Theorem~1]{Gol66} it follows that 
\begin{equation*}
|\Phi'(\xi)|\geq\mcap \Gamma\left(1-\frac{1}{|\xi|^2}\right) \quad \mbox{ for any }\xi\in\myh\CC\setminus\myo\DD(0,1).
\end{equation*}

Let $\varepsilon'\in(0,\varepsilon)$. Let us estimate $\dist((\Gamma)_{\varepsilon'}, (\Gamma)_\varepsilon)$ Let $z_1\in(\Gamma)_{\varepsilon'}, z_2\in(\Gamma)_{\varepsilon}$ be such that $\dist((\Gamma)_{\varepsilon'}, (\Gamma)_\varepsilon)=\dist(z_1,z_2)$. Let $L=\varphi([z_1,z_2])$. Then
\begin{align*}
\dist(z_1,z_2)&=\int\limits_{[z_1,z_2]}|dz|=\int\limits_L|\Phi'(\xi)||d\xi|\geq\mcap\Gamma\int\limits_L\left(1-\frac{1}{|\xi|^2}\right)|d\xi|\\
&\geq\mcap\Gamma\int\limits_{1+\varepsilon'}^{1+\varepsilon}\left(1-\frac{1}{r^2}\right)dr=(\varepsilon-\varepsilon')\left(1-\frac{1}{(1+\varepsilon')(1+\varepsilon)}\right)\mcap\Gamma.
\end{align*}
Since $\dist(\Gamma, (\Gamma)_\varepsilon)>\dist((\Gamma)_{\varepsilon'}, (\Gamma)_\varepsilon)$ we can tend $\varepsilon'\to 0$ and get 
\begin{equation*}
\dist(\Gamma, (\Gamma)_\varepsilon)\geq\frac{\varepsilon^2}{(1+\varepsilon)}\mcap \Gamma.
\end{equation*}
\end{proof}

The following simple corollary of Lemma~\ref{L3} in the case of a compact set with a finite number of connected components is exactly~\cite[Lemma~2]{PeRa} (see also~\cite[Section~3]{AptYa}). Actually our proof is analogous to the proof from~\cite{PeRa}.

\begin{Corollary}
\label{Cor_dist}
Let $K$ be a compact set such that the capacity of each its connected component is not less than $\delta$ for some $\delta>0$. Then 
\begin{equation*}
\dist(K, (K)_\varepsilon)\geq\frac{\varepsilon^2}{(1+\varepsilon)}\delta 
\end{equation*}
for any $\varepsilon>0$.
In particular, for $\varepsilon\in(0,1)$ we have  $\dist(K, (K)_\varepsilon)>\frac{\varepsilon^2}{2}\delta$.
\end{Corollary}
\begin{proof}
Let $z_1\in K, z_2\in (K)_\varepsilon$ be such that $\dist(K, (K)_\varepsilon)=\dist(z_1, z_2)$. Let $K_1$ be the connected component of $K$ that contains $z_1$. Since $K_1\subset K$ we have $g_{\myh\CC\setminus K_1}(z)\geq g_{\myh\CC\setminus K}(z)$ for all $z\in\CC$. In particular, $g_{\myh\CC\setminus K_1}(z_2)\geq g_{\myh\CC\setminus K}(z_2)=\log(1+\varepsilon)$. Therefore $\dist(K_1, z_2)\geq\dist(K_1, (K_1)_\varepsilon)$. Applying Lemma~\ref{L3} we obtain
\begin{equation*}
\dist(z_1, z_2)\geq\dist(K_1, z_2)\geq\dist(K_1, (K_1)_\varepsilon)
\geq\frac{\varepsilon^2}{(1+\varepsilon)}\mcap K_1 
\geq\frac{\varepsilon^2}{(1+\varepsilon)}\delta.
\end{equation*}
\end{proof}

Recall that the Hausdorff distance $d_H$ between two non-empty compact sets $X, Y\subset\CC$ is
\begin{equation*}
d_H(X,Y)=\inf\{\varepsilon: X\subset Y^\varepsilon \mbox{ and }Y\subset X^\varepsilon\}.
\end{equation*}
We denote the convergence in Hausdorff topology by $\xrightarrow{H}$.
It is well known that the set of all non-empty compact subsets of any compact set $K\subset\CC$ is compact in Hausdorff topology. Therefore from any sequence of non-empty compact sets contained in some fixed disc $\myo\DD(0, R)$ we can choose a subsequence that converges in Hausdorff topology.

The following Statement~\ref{St_H_cap} shows that under some regularity condition convergence in Hausdorff topology implies convergence of capacities. 
Analogous statement in the case when the number of connected components of all compact sets is finite and less than some fixed number $N$ was proved in~\cite[Theorem~2]{PeRa} (see also~\cite[Section~3]{AptYa}).

\begin{Statement}
\label{St_H_cap}
Let $K_n$, $n\in\NN$, be compact sets and capacity of any connected component of each $K_n$ be not less than some $\delta>0$. Let $K_n\xrightarrow{H}K$ for some compact set $K\subset\CC$. Then there exists $\lim_{n\to\infty}\mcap{K_n}=\mcap K$.
\end{Statement}

\begin{Remark}
In general case, for arbitrary compact sets $K_n$ such that $K_n\xrightarrow{H}K$ we have only $\limsup_{n\to\infty}\mcap{K_n}\leq\mcap K$. This inequality easily follows from the definition of the capacity via the transfinite diameter, for details see, e.g.,~\cite[Section~2]{KaKo22}. It is easy to construct an example when $\mcap K_n=0$, $K_n\xrightarrow{H}K$, and $\mcap K=1$. For example, for $K_n$ we can take all rational points of the form $k/2^n$, $k\in\ZZ$, belonging to the interval $[-2,2]$. Then $K=[-2,2]$. More complicated example, which goes back to Ahlfors and Beurling, when $K_n\xrightarrow{H}K$, $\mcap K =1$ and the limit of $\mcap K_n$ can be made an arbitrary number from the interval $(0,1)$, see in~\cite[Proposition 4.1]{KaKo22}.
\end{Remark}

Before the proof of Statement~\ref{St_H_cap} we prove auxiliary simple Lemma~\ref{L_H}.

We denote the diameter of any set $G$ by $\diam G$.

\begin{Lemma}
\label{L_H}
Let $K_n$, $n\in\NN$, be compact sets and diameter of any connected component of each $K_n$ be not less than some $\delta>0$. Let $K_n\xrightarrow{H}K$ for some compact set $K\subset\CC$. Then the diameter of any connected component of $K$ is also not less than $\delta$.
\end{Lemma}
\begin{proof}
Let us assume the opposite. Then there exists a connected component $F$ of $K$ such that $\diam F<\delta$. Let $\rho:=(\delta-\diam F)/2>0$. 
Then (see~\cite[Corollary after Theorem~21]{Al77}) there exists a neighbourhood $U$ of $F$ such that $U\subset F^\rho$ and $\partial U\cap K=\varnothing$.
Let $\varepsilon:= \dist(\partial U, K)/2>0$. Then $K^\varepsilon\cap \partial U =\varnothing$. It immediately follows from the definition of convergence in Hausdorff topology that there exists such $n$ that $K\subset K_n^{\varepsilon}$ and $K_n\subset K^{\varepsilon}$. From the first inclusion it follows that $F^\varepsilon\cap K_n\ne\varnothing$. Let $z\in F^\varepsilon\cap K_n$. The connected component $K_{n, z}$ of $z$ in $K_n$ has the diameter not less than $\delta$. Since $z\in U$ and $\diam U\leq\diam F+\rho<\delta$, we have $K_{n, z}\cap\partial U\ne\varnothing$. But $K^\varepsilon\cap \partial U =\varnothing$, and we get the contradiction to $K_n\subset K^{\varepsilon}$.
\end{proof}

\begin{proof}[Proof of Statement~\ref{St_H_cap}]
For any connected compact set $F$ we have (see~\cite[Theoren 5.3.2]{Ra95}) $\diam F/4\leq\mcap F\leq\diam F$.
Therefore the diameter of any connected component of each $K_n$ is not less than $\delta$. Since $K_n\xrightarrow{H}K$, from Lemma~\ref{L_H} we have that the diameter of any connected component of $K$ is also not less than $\delta$. Consequently, the capacity of any connected component of $K$ is not less than $\delta/4$. Replacing $\delta$ by $\delta/4$, further we assume that the capacity of any connected component of $K$ is not less than $\delta$.

Let $\varepsilon\in (0,1)$. Since $K_n\xrightarrow{H}K$, there exists $N\in\NN$ such that $K_n\subset K^{\delta\varepsilon^2/2}$  and $K\subset K_n^{\delta\varepsilon^2/2}$ for all $n>N$. It follows from Corollary~\ref{Cor_dist} that $K^{\delta\varepsilon^2/2}\subset \myh{(K)_\varepsilon}$ and $K_n^{\delta\varepsilon^2/2}\subset \myh{(K_n)_\varepsilon}$. Since 
\begin{equation*}
g_{\myh\CC\setminus\myh{(K)_\varepsilon}}(z) =\max\{g_{\myh\CC\setminus K}(z)-\log(1+\varepsilon), 0\},
\end{equation*}
it follows from the definitions~\eqref{E(K)},~\eqref{V^lambda},~\eqref{g(z)} that $\mcap \myh{(K)_\varepsilon}=(1+\varepsilon)\mcap K$. Similarly, $\mcap \myh{(K_n)_\varepsilon}=(1+\varepsilon)\mcap K_n$. Since $K_n\subset K^{\delta\varepsilon^2/2}\subset \myh{(K)_\varepsilon}$, we have $\mcap K_n\leq\mcap \myh{(K)_\varepsilon}=(1+\varepsilon)\mcap K$ for all $n>N$. Similarly, $\mcap K\leq\mcap \myh{(K_n)_\varepsilon}=(1+\varepsilon)\mcap K_n$ for all $n>N$. Thus
\begin{equation*}
(1+\varepsilon)^{-1}\mcap K\leq\mcap K_n\leq(1+\varepsilon)\mcap K.
\end{equation*}
Tending $\varepsilon$ to 0, we get the statement.
\end{proof}

The following simple Proposition~\ref{Prop_H} shows that the space of all admissible compact sets $\KKf$ for any fixed germ $f_\infty$ is closed (up to taking the polynomially convex hull) in Hausdorff topology. We need to take the polynomially convex hull just because of the definition~\ref{Def_KKf} of $\KKf$.

\begin{Proposition}
\label{Prop_H}
Let $f_\infty$ be an analytic germ at $\infty$ and $K_n\in\KKf$ be such that $K_n\xrightarrow{H}K$ for some compact set $K\subset\CC$. Then $\myh K\in\KKf$.
\end{Proposition}
\begin{proof}
Let us suppose that $\myh K\notin\KKf$. Then there exists a closed path $\gamma\subset K_\infty$ with the initial and the final points at $\infty$ such that either the analytic continuation of $f_\infty$ along $\gamma$ is impossible or it gives a germ different from $f_\infty$. Since $K_n\in\KKf$, we have $K_n\cap\gamma\ne\varnothing$ for all $n$. Since $K_n\xrightarrow{H}K$, there exists $n_0$ such that $K_{n_0}\subset K^\varepsilon$ for $\varepsilon=\dist(K, \gamma)/2$. Consequently, $K_{n_0}\cap \gamma=\varnothing$. Contradiction.
\end{proof}

The following Lemma~\ref{L_cut} shows (for details see \S\ref{s4}) that in Theorem~\ref{t1} we can consider only compact sets $K$ such that $K\supset E$ and all  connected components of $K$ intersect $E$. 

\begin{Lemma}
\label{L_cut}
Let $E$ be a compact set in $\CC$. Let $f_\infty\in\AA(\myh\CC\setminus E)$. Let $K\in\KKf$ and $K\supset E$. 
Then the union $F$ of all connected components of $K$ that intersect $E$ also belongs to $\KKf$.
\end{Lemma}
\begin{proof}
First of all, we show that $F$ is a compact set. Let us assume the opposite. Then there exists a sequence of points $x_n\in F$, $n\in\NN$, such that $x_n\to x_0$ and $x_0\in K\setminus F$. Let $K_0$ be the connected component of $x_0$ in $K$. Then $\varepsilon:=\dist(K_0, E)>0$. There exists (see~\cite[Corollary after Theorem~21]{Al77}) a neighbourhood $U$ of $K_0$ such that $ U\subset K_0^{\varepsilon/2}$ and $\partial U\cap K = \varnothing$. Since $x_n\to x_0$, there exists $n_0$ such that $x_{n_0}\in U$. Since $x_{n_0}\in F$, its connected component $F_{n_0}$ intersects $E$. But $x_{n_0}\in U$ and $E\cap\myo U=\varnothing$ by construction. Therefore $F_{n_0}\cap\partial U\ne\varnothing$. Contradiction.

Now let us prove the statement of the Lemma. Let $\GGG$ be the set of all connected components of $K$ that don't intersect $E$. Put $G:=K\setminus F=\cup_{G_\alpha\in\GGG}G_\alpha$. Of course, $\GGG$ may be uncountable and G may be not a compact set.
For each $G_\alpha\in\GGG$, we have $\dist(G_\alpha, F)>0$. 
Applying Lemma~\ref{L0}, we choose a simply connected neighbourhood $V_\alpha$ of $G_\alpha$ such that $\myo{ V_\alpha}\cap F=\varnothing$. 
Let $\rho_\alpha:=\dist(F, \partial V_\alpha)>0$.
Applying again~\cite[Corollary after Theorem~21]{Al77} we choose a neighbourhood $W_\alpha$ such that $\myo{W_\alpha}\subset G_\alpha^{\rho_\alpha/2}$, $G_\alpha\subset W_\alpha\subset V_\alpha$ and $\partial W_\alpha\cap K=\varnothing$.
Let $\delta_\alpha:=\dist(\partial W_\alpha, K)>0$. 
Put $\varepsilon_\alpha=\min\{\rho_\alpha, \delta_\alpha\}$.
Let $\Phi_{\alpha,k}$, $k=1, \dots, N_\alpha$, be a finite subcover of $\varepsilon_\alpha/2$-cover $\partial W_\alpha$. Then the exterior boundary of the closure $\myo{\bigcup_{k=1}^{N_\alpha}\Phi_{\alpha,k}}$ is a closed Jordan curve $\gamma_\alpha$. Let us denote the bounded connected component of $\myh\CC\setminus\gamma_\alpha$ by $U_\alpha$. We have that $U_\alpha\subset V_\alpha$, $\myo{U_\alpha}\cap F = \varnothing$ and $\dist(\gamma_\alpha, K)\geq \varepsilon_\alpha/2$.
The union of $U_\alpha$ for all $G_\alpha\in\GGG$ is an open cover of $G$. Then there exists a countable subcover $U_n, n\in\NN$,  of $G$, because any separable metric space, in particular the complex plane, is a Lindel\"of space. For each $n\in\NN$ we put 
\begin{equation*}
F_n:= K\setminus \bigcup\limits_{k=1}^n U_k
\end{equation*}
and $F_0:=K$. Since $\partial U_n\cap K=\varnothing$, all $F_n$ are compact sets. Moreover, $F_n\supset F_{n+1}$ and $F=\cap_{n=0}^\infty F_n$. 

Let us prove by induction that $F_n\in\KKf$ for all $n\in\NN$. For $n=0$ we have $F_0=K\in\KKf$. Let us assume that $F_n\in\KKf$ for some $n\in\NN\cup 0$ and show that $F_{n+1}\in\KKf$. Indeed. 
Since $\myh\CC\setminus F_n$ is connected and $\partial U_{n+1}\cap F_n=\varnothing$, the complement of $F_{n+1}=F_n\setminus U_{n+1}$ is also connected. 
Let $\gamma$ be an arbitrary closed path in $\myh\CC\setminus F_{n+1}$ with the initial and the final points at $\infty$. 
In particular, $\gamma\subset\myh\CC\setminus E$.
Therefore there exists the analytic continuation of $f_\infty$ along $\gamma$. We need to show that the analytic continuation of $f_\infty$ along $\gamma$ gives $f_\infty$.
Let us choose a point $z_0\in U_{n+1}\setminus\gamma$.
Since $U_{n+1}$ is simply connected, by Riemann mapping theorem there exists a conformal map $\varphi$ of $U_{n+1}$ onto the unit disc $\DD(0,1)$ that maps $z_0$ to 0. By Caratheodory theorem, $\varphi$ extends to the homeomorphism of $\myo{U_{n+1}}$ and $\myo\DD(0,1)$. Let $\myt\gamma:=\gamma\cap \myo{U_{n+1}}$ and $\pr(\varphi(\myt\gamma))$ be the radial projection of $\varphi(\myt\gamma)$ on the unit circle $\TT(0,1)$. 
Let $\breve\gamma$ be the closed path with the initial and the final points at $\infty$ that coincides with $\gamma$ in $\myh\CC\setminus\myo{U_{n+1}}$ and with $\phi^{-1}(\pr(\varphi(\myt\gamma)))$ on $\gamma_{n+1}=\partial U_{n+1}$. We have $\breve\gamma\subset\myh\CC\setminus F_n$ by construction. Consequently, there exists the analytic continuation of $f_\infty$ along $\breve\gamma$ and as a result of it we obtain the same germ $f_\infty$. 
Since $\myo{U_{n+1}}\cap E=\varnothing$ and $U_{n+1}$ is simply connected, 
for any germ $f_{\xi}$ at any point $\xi\in\myo{U_{n+1}}$ that is the analytic continuation of $f_\infty$ along some path in $\myh\CC\setminus E$
there exists (single-valued) meromorphic continuation of $f_{\xi}$ on whole $\myo{U_{n+1}}$. Consequently, since $\gamma$ and $\breve\gamma$ coincide outside $\myo{U_{n+1}}$, the analytic continuations of $f_\infty$ along $\breve\gamma$ and $\gamma$ coincide. Therefore the analytic continuation of $f_\infty$ along $\gamma$ gives $f_\infty$.

So, $F_n\in\KKf$ for all $n\in\NN$.
Since $F_n\supset F_{n+1}$, $F=\cap_{n=0}^\infty F_n$ and $\myh\CC\setminus F_n$ is connected, we have $F_n\xrightarrow{H} F$ and $\myh\CC\setminus F$ is also connected. Thus, from Proposition~\ref{Prop_H} it immediately follows that $F\in\KKf$.
\end{proof}

It is well known that capacity is not semi-additive with any constant, see classical example~\cite[Corollary~5.2.6]{Ra95}. 
The following Lemma~\ref{L_cap} gives in some sense the property of continuity of capacity.

\begin{Lemma}
\label{L_cap}
For any $\beta\in(0,1)$ and $\varepsilon>0$ we put 
\begin{equation*}
\delta=\delta(\varepsilon,\beta):=e^{-\frac{\log^2\beta}{\log{(1+\varepsilon)}}}>0.
\end{equation*}
Let $A, B\subset\DD(0,1)$ be two compact sets such that $\mcap A<\delta$ and $\mcap B>\beta$. Then $\mcap(A\cup B)\leqslant (1+\varepsilon)\mcap B$.
\end{Lemma}
\begin{proof}
Let $\gamma_A,\gamma_B,\gamma_{A\cup B}$ be the Robin constant of $A, B, A\cup B$, respectively. Since $A, B\in\DD(0,1)$ the following inequality holds true (see~\cite[Theorem~5.1.4]{Ra95}): $1/\gamma_{A\cup B}\leq1/\gamma_A+1/\gamma_B$. Therefore $\gamma_{A\cup B}\geq \frac{\gamma_A\gamma_B}{\gamma_A+\gamma_B}$. So, we have
\begin{equation*}
\mcap (A\cup B)=e^{-\gamma_{A\cup B}}\leq e^{-\frac{\gamma_A\gamma_B}{\gamma_A+\gamma_B}}=e^{-\gamma_B}e^{\frac{\gamma_B^2}{\gamma_A+\gamma_B}}=\mcap B\cdot e^{\frac{\gamma_B^2}{\gamma_A+\gamma_B}}
\end{equation*}
Since $A, B\subset\DD(0,1)$, we have $\gamma_A,\gamma_B,\gamma_{A\cup B}>0$. Therefore
\begin{equation*}
\mcap (A\cup B)\leq\mcap B\cdot e^{\frac{\gamma_B^2}{\gamma_A}}\leq \mcap B\cdot e^{(-\log \beta)^2\cdot\frac{\log(1+\varepsilon)}{\log^2\beta}}=(1+\varepsilon)\mcap B.
\end{equation*}
\end{proof}

\section{Proof of Theorem~\ref{t1}}
\label{s4}

Recall that we denote $c_0:=\inf\limits_{K\in\KKf}\mcap K$. If $c_0=0$ the statement of Theorem~\ref{t1} immediately follows from Statement~\ref{St_c_0}. Therefore, further we assume $c_0>0$.
Let $K_n, n\in\NN$, be a minimizing sequence, i.e. $K_n\in\KKf$ and
\begin{equation}
\label{c_00}
\lim_{n\to\infty}\mcap K_n = \inf\limits_{K\in\KKf}\mcap K =: c_0>0.
\end{equation}
From Corollary~\ref{Cor_D} it follows that we can assume that all $K_n$ and $E$ belong to some big disc $\DD(0,r)$. Using scaling, we assume that $r=1/3$, i. e. $K_n\subset\DD(0,1/3)$ for all $n\in\NN$ and $E\subset\DD(0,1/3)$. 
Since $\mcap E=0$, we have $\mcap K_n=\mcap(K_n\cup E)$. Therefore we assume that $K_n\supset E$ for all $n\in\NN$. Now, using Lemma~\ref{L_cut} we assume that all connected components of each $K_n$ intersect $E$.

For each $k\in\NN$ we denote the closure of $(1/3k)$-neighbourhood of $E$ by $\myo{E^k}$, i. e.
\begin{equation*}
\myo{E^k}: =\{z\in\myh\CC: \dist(z, E)\leq 1/3k\}.
\end{equation*}
We denote $K_{n,k}:=K_n\cup \myo{E^k}$. Let $\myh {K_{n,k}}$ be the polynomially convex hull of $K_{n,k}$, i.e. the complement to the unbounded connected component of $\myh\CC\setminus K_{n,k}$. Since $\myh\CC\setminus\myh {K_{n,k}}$ is connected and $\myh {K_{n,k}}\supset K_n$, we have $\myh {K_{n,k}}\in\KKf$.
Since the set of all non-empty compact subsets of any compact set (in our case of $\myo\DD(0,2/3)$) is compact in Hausdorff topology, using a diagonal process we can choose such subsequence $\Lambda$ of $\NN$ that for each $k\in\NN$
\begin{equation*}
\myh {K_{n,k}}\xrightarrow{H} F_k, \quad n\in\Lambda,
\end{equation*}
where $F_k$ are some compact sets in $\myo\DD(0,2/3)$.
From Proposition~\ref{Prop_H} it follows that $\myh{F_k}\in\KKf$, where $\myh{F_k}$ is the polynomially convex hull of $F_k$.
Since $\myh {K_{n,k}}\supset\myh {K_{n,k+1}}$, we have $F_k\supset F_{k+1}$.
Since all connected components of $K_n$ intersect $E$, each connected component of $K_{n,k}$ contains a disc of radius $1/3k$. Therefore the capacity of any connected component of $\myh{K_{n, k}}$ is not less than $1/3k$. Applying Statement~\ref{St_H_cap}, we get that for each $k\in\NN$ there exists 
\begin{equation}
\label{lim_Knk}
\lim\limits_{n\in\Lambda}\mcap \myh{K_{n, k}} = \mcap F_k =\mcap \myh{F_k}.
\end{equation}
Let us denote
\begin{equation*}
F:=\bigcap\limits_{k=1}^\infty \myh{F_k}.
\end{equation*}
Since $\myh\CC\setminus\myh{F_k}$ is connected, we also have that $\myh\CC\setminus F$ is connected, i.e. $F=\myh F$. It immediately follows from the definition of $F$ that $\myh F_k\xrightarrow{H} F$. Therefore applying again Proposition~\ref{Prop_H}, we get $F\in\KKf$. 

Let us show that $\mcap F =c_0$. Let us choose arbitrary $\varepsilon>0$ and apply Lemma~\ref{L_cap} for $\varepsilon$ and $\beta=c_0$. Let $\delta=\delta(\varepsilon, c_0)$ from Lemma~\ref{L_cap}. Thus, $\delta=\delta(\varepsilon)$ depends only on $\varepsilon$, because $c_0$ is fixed. Since $\myo{E_k}\supset \myo{E_{k+1}}$, $E=\bigcap_{k=1}^{\infty} \myo{E^k}$ and $\myo{E^k}$ are compact sets, there exists (see~\cite[Theorem 5.1.3]{Ra95})
\begin{equation*}
\lim\limits_{k\to\infty}\mcap \myo{E^k} = \mcap E =0.
\end{equation*}
Consequently, there exists $K(\varepsilon)$ such that $\mcap \myo{E^k}<\delta(\varepsilon)$ for all $k\geq K(\varepsilon)$. Since $\mcap K_n>c_0$, from Lemma~\ref{L_cap} it follows that
\begin{equation*}
\mcap\myh{K_{n,k}}=\mcap K_{n,k} \leq (1+\varepsilon) \mcap K_n
\end{equation*}
for $n\in\NN$ and $k\geq K(\varepsilon)$. Applying~\eqref{c_00} and~\eqref{lim_Knk}, we get $\mcap \myh{F_k}\leq (1+\varepsilon)c_0$ for all $k\geq K(\varepsilon)$.
Since $F=\bigcap\limits_{k=1}^\infty \myh{F_k}$, we have $\mcap F\leq(1+\varepsilon) c_0$. Due to the arbitrary choice of $\varepsilon>0$, we obtain $\mcap F = c_0$.

{\bf Aleksandr V. ~Komlov}

Steklov Mathematical Institute of RAS, 

Moscow, Russia

{\it E-mail}: komlov@mi-ras.ru

\end{document}